\documentclass[reqno, 12pt]{amsart}

\usepackage{amsmath,amsfonts,amssymb,amscd,amsthm,amsbsy,bm,latexsym}
\usepackage{gensymb}
\usepackage[pdftex,bookmarksnumbered=true,bookmarksopen=true,          
            colorlinks=true,pdfborder=001,citecolor=blue,
            linkcolor=red,anchorcolor=green,urlcolor=blue]{hyperref}
\usepackage{graphicx}
\usepackage{paralist}
\usepackage{calc}

\allowdisplaybreaks

\usepackage{color}

\newcommand{\Proof}{\begin{proof}}
\newcommand{\End}{\end{proof}}

\newtheorem{lemma}{Lemma}[section]
\newtheorem{theorem}{Theorem}[section]
\newtheorem{definition}{Definition}[section]
\newtheorem{proposition}{Proposition}[section]
\newtheorem{remark}{Remark}[section]
\newtheorem{corollary}{Corollary}[section]

\numberwithin{equation}{section}

\begin{document}

\title[representation formula of viscosity solutions]{A representation formula of viscosity solutions to weakly coupled systems of Hamilton-Jacobi equations with applications to regularizing effect}

\author[L. Jin]{Liang Jin}
\address[L. Jin]{Department of Mathematics, Nanjing University of Science and Technology, Nanjing 210094, China}
\email{jl@njust.edu.cn}

\author[L. Wang]{Lin Wang}
\address[L. Wang]{Yau Mathematical Sciences Center, Tsinghua University, Beijing 100084, China}
\email{lwang@math.tsinghua.edu.cn}

\author[J. Yan]{Jun Yan}
\address[J. Yan]{School of Mathematical Sciences, Shanghai 200433, China}
\email{yanjun@fudan.edu.cn}

\date{\today}

\begin{abstract}
  Based on a fixed point argument, we give a {\it dynamical representation} of the viscosity solution to Cauchy problem of certain weakly coupled systems of Hamilton-Jacobi equations with continuous initial datum. Using this formula, we obtain some regularity results related to the viscosity solution, including a partial extension of Lions' regularizing effect \cite{L} to the case of weakly coupled systems.
\end{abstract}


\thanks{The first author is supported by the NSF of China (Grants 11571166) and Startup Foundation of Nanjing University of Science and Technology. The second author is supported by the NSF of China (Grant 11631006, 11790273). The third author is supported by the NSF of China (Grant 11790273, 11631006).}

\maketitle
\thispagestyle{empty}
\tableofcontents

\section{Introduction and main results}
In this paper, we shall consider the following Cauchy problem of the weakly coupled system of evolutionary Hamilton-Jacobi equations
\begin{equation}\label{HJs}
\left\{
  \begin{array}{ll}
    \partial_{t}u_{i}+H_{i}(x,\partial_{x}u_{i},\mathbf{u})=0, & \hbox{$(x,t)\in M\times[0,\infty)$;} \\
    \mathbf{u}(x,0)=\mathbf{\varphi}(x), & \hbox{$x\in M, 1\leq i\leq m$.}
  \end{array}
\right.
\end{equation}
where $\mathbf{u}$ (resp. $\varphi$) are continuous vector-valued functions from $M\times[0,\infty)$ (resp. $M$) to $\mathbb{R}^{m}$. Hereinafter, the notation $u_{i}$ always denotes the $i$-th component of the vector $\mathbf{u}\in\mathbb{R}^{m}$.

In recent years, motivated by optimal switching problems, there have been many studies on the viscosity solutions to the system \eqref{HJs}, including the corresponding extensions of the weak KAM and Aubry-Mather theories \cite{DS,DSZ,FGM,MSTY}, the large-time behavior of solutions \cite{MT1,MT2,N}, and homogenization problems \cite{CLL,MT3}. Most of the studies focus on the linearly coupled case, i.e. Hamiltonians $H_{i}$ linearly depend on $\mathbf{u}\in\mathbb{R}^{m}$ (the linear coefficients may depend on $x$ variable). In these cases, well posedness results in \cite{EL,IK} can be applied since the monotonicity assumption is satisfied. Generally speaking, the linearly coupled case could be seen as a direct generalization of classical Hamilton-Jacobi equation, i.e., $H=\bar{H}(x,p)$ or discounted Hamilton-Jacobi equation, i.e., $H=\lambda u+\bar{H}(x,p)$when $m=1$.

From this point of view, the nonlinear weakly coupled system \eqref{HJs} could be considered as a natural extension of contact Hamilton-Jacobi equations studied in a series of works \cite{SWY,WWY1,WWY2}. The contact Hamilton-Jacobi equations are functional dual of contact Hamiltonian systems,  which have various physical applications in different areas, for instance, geometric optics and wave propagations, see \cite{Ar,BCT1}. We note that our work begins with non-monotonicity assumptions (H1)-(H3) (see Section 2.2 below).

First of all, we establish an implicit variational principle (also referred as dynamical programing principle), from which we introduce the notions of the variational solution and the solution semigroup for (\ref{HJs}). Moreover, by showing the equivalence between variational solutions and viscosity solutions, we obtain the dynamical representation of the viscosity solution of (\ref{HJs}). More precisely, we have
\begin{theorem}\label{the1}
If $H_{i}, 1\leq i\leq m$ satisfies (H1)-(H3), the Cauchy problem of the weakly coupled system  (\ref{HJs}) has a unique viscosity solution $\mathbf{u}(x,t)\in C(M\times[0,\infty),\mathbb{R}^{m})$, which can be represented by
\begin{equation}\label{vas eq11}
u_{i}(x,t)=\inf_{\substack{\gamma_i(t)=x\\ \gamma_i\in C^{ac}([0,t],M)}}\bigg\{\varphi_{i}(\gamma_i(0))+\int^{t}_{0}L_{i}(\gamma_i(s),\dot{\gamma}_i(s),\mathbf{u}(\gamma_i(s),s))ds\bigg\},
\end{equation}
where $L_{i}:TM\times\mathbb{R}^{m}\rightarrow\mathbb{R}$ denotes the convex dual associated to $H_{i}$.
\end{theorem}

\begin{remark}
When we finish writing this paper, we are told that in a previous work \cite{LC}, the authors obtained a similar result as Theorem \ref{the1} in the case that the Hamiltonians $H_{i}$ is independent of $x$.
\end{remark}

The regularity of viscosity solutions plays a fundamental role in applying the dynamical approach to the study of both classical and contact Hamilton-Jacobi equations, see \cite{Fa,SWY} for instance. The same situation occurs in the study of systems (\ref{HJs}).

Note that when $m=1$, a typical phenomenon found by P.L.Lions and other PDE specialists is that assume the initial data $\mathbf{\varphi}(x)$ is only {\it continuous }, then the solution $u(x,t)$ to \eqref{HJs} is locally Lipschitz on $M\times(0,T)$. We believe that similar phenomenon also holds true for weakly coupled systems \eqref{HJs}.

In order to avoid technical difficulties, we restrict ourselves to a typical model.  More precisely, we work under (H1)-(H4) and the additional assumption:
\begin{itemize}[\bf (H*)]
  \item For $1\leq i\leq m$, $H_{i}$ can be divided into two parts, i.e.
        \begin{equation}
        H_{i}(x,p,\mathbf{u})=h_{i}(x,p)+P_{i}(x,\mathbf{u}),
        \end{equation}

  and there exists $r,A\in(1,\infty)$ such that for $1\leq i\leq m$,
        \begin{equation}
        \begin{split}
        A^{-1}|p|^{r}-A\leq h_{i}(x,p)\leq A|p|^{r},\\
        |\partial_{x}h_{i}(x,p)|\leq A(|p|^{r}+1),\\
        \end{split}
        \end{equation}
\end{itemize}

\begin{remark}
One notice that the Hamiltonians satisfying (H*) can be viewed as a generalization of classical mechanical system, where $h_{i}(x,p)$ denotes the kinetic energy, usually represented by a Riemannian metric, and $P_{i}$ denotes the potential.
\end{remark}

Let $\Theta$ be the maximal Lipschitz constant of $H_i$ with respect to $\mathbf{u}$ for $1\leq i\leq m$, we obtain
\begin{theorem}\label{Thm Lip11}
Let $H_{i},1\leq i\leq m$ be Hamiltonians satisfying (H1)-(H3) and (H*), $\mathbf{u}$ the unique viscosity solution associated to the Cauchy problem \eqref{HJs}. Then $\mathbf{u}$ is locally Lipschitz continuous with respect to $t$ and $x$ on $M\times(0,\infty)$. Moreover, there is a continuous function $\kappa_{\varphi}:(0,\infty)\rightarrow(0,\infty)$ such that for any $x,y\in M$ and $t\in(0,\infty)$,
$$
\|\mathbf{u}(x,t)-\mathbf{u}(y,t)\|\leq\kappa_{\varphi}(t)\cdot|x-y|.
$$
Moreover, there exist constants $t_{\Theta},\kappa_{0}>0$ such that for any $t\in (0,t_{\Theta}], \kappa_{\varphi}(t)\leq\kappa_{0}\cdot t^{-\frac{1}{r}}$.
\end{theorem}

\begin{remark}
To handle the case of $m>1$, a new priori estimate has to be established. Unfortunately, we could not use this type of estimate to treat all systems satisfying (H1)-(H3). Thus Theorem \ref{Thm Lip11} partially generalizes the regularizing effect result for the case of the single equation, see \cite{Fa},\cite{L}.
\end{remark}

Based on Theorem \ref{Thm Lip11}, we obtain some further properties, including the semiconcavity of the viscosity solutions and the regularity of the action minimizing curves.
\begin{corollary}\label{cur lip11}
Let $\mathbf{u}$ be the unique viscosity solution associated to the Cauchy problem \eqref{HJs} and for each $1\leq i\leq m$, $\xi_{i}:[0,t]\rightarrow M$ be an absolutely continuous curve with $\xi_{i}(t)=x$ such that
$$
u_{i}(x,t)=\varphi_{i}(\xi_{i}(0))+\int^{t}_{0}L_{i}(\xi_{i}(s),\dot{\xi}_{i}(s),\mathbf{u}(\xi_{i}(s),s))ds,
$$
then there hold
 \begin{enumerate}[(i)]
   \item $u_{i}$ is locally semiconcave on $M\times(0,\infty)$.

   \item $\xi_{i}$ is locally Lipschitz on $(0,t]$.

   \item for $s\in(0,t)$, if $u_{i}$ is differentiable at $(\xi(s),s)$, then $\xi_{i}$ is differentiable at $s$. Denote by $V_i$ the derivative of $\xi_{i}$ at $s$ and let $P_i=\partial_{x}u_{i}(\xi_{i}(s),s)$, we have
         \begin{equation}\label{dual11}
         \langle P_i,V_i \rangle=L_{i}(\xi_{i}(s),V_i,\mathbf{u}(\xi_{i}(s),s))+H_{i}(\xi_{i}(s),P_i,\mathbf{u}(\xi_{i}(s),s)).
         \end{equation} Or equivalently,
         \[V_i=\frac{\partial H_i}{\partial p}(\xi_{i}(s),P_i,\mathbf{u}(\xi_{i}(s),s)).\]
         Similar conclusion holds for $s=t$ with the only difference that $\xi_{i}$ can only have derivative from left side.
\end{enumerate}
\end{corollary}

The last statement of this corollary asserts that the minimizing curves passing through the differentiable points of $\mathbf{u}$ satisfy certain ODEs. Comparably, for the contact Hamilton-Jacobi equation corresponding to $m=1$ in (\ref{HJs}), we know that the action minimizing curve $\xi_i$ is smooth enough and it is exactly the $x$-component of the characteristics (contact Hamilton equations) if $H$ is of class $C^3$ \cite{WWY1}:
\begin{align*}
\left\{
        \begin{array}{l}
        \dot{x}=\frac{\partial H}{\partial p}(x,u,p),\\
        \dot{p}=-\frac{\partial H}{\partial x}(x,u,p)-\frac{\partial H}{\partial u}(x,u,p)p,\qquad (x,p,u)\in T^*M\times\mathbb{R},\\
        \dot{u}=\frac{\partial H}{\partial p}(x,u,p)\cdot p-H(x,u,p).
         \end{array}
         \right.
\end{align*}
Unfortunately, the notion of characteristics does not carry over directly to the case when $m\geq 2$. And it seems that the Lipschitz regularity of $\xi_i$ can not be improved.

This paper is outlined as follows. In Section 2, we introduce the notations and general settings as preliminaries. In Section 3, we use a version of Picard iteration to define the variational solution and the solution semigroup. In Section 4, by showing the equivalence between viscosity solutions and variational solutions, we will complete the proof of Theorem \ref{the1}. In Section 5 and Section 6, we will discuss the locally Lipschitz continuity of the viscosity solution and further benefits from this property. The proofs of Theorem \ref{Thm Lip11} and Corollary \ref{cur lip11} are completed in these final sections.

\section{Preliminaries}

\subsection{Notations}
In this part, we fix notations for later presentations. Once and for all, $n,m>1$ are two fixed positive integers and $|\cdot|$ denotes the Euclidean norm.

\vspace{0.5cm}

Let $M$ be $\mathbb{T}^{n}$, the $n$-dimensional flat torus with the standard product metric. For any two points $x,y$ on $M$, we also use $|x-y|$ to denote their distance induced by the flat metric on $M$. Let $\|\cdot\|$ denote the $L^{\infty}$-norm  for vectors in $\mathbb{R}^{m}$, i.e. for any $v\in\mathbb{R}^{m}$,
$$
\|v\|=\max_{1\leq i\leq m}|v_{i}|.
$$
Denote by $TM$ the tangent bundle of $M$ and by $(x,\dot{x})$ a point of $TM$, where $x\in M$ and $\dot{x}\in T_{x}M=\mathbb{R}^{n}$; denote by $T^{\ast}M$ the cotangent bundle of $M$ and by $(x,p)$ a point of $T^{\ast}M$, where $p\in\mathbb{R}^{n}$ a linear form on $T_{x}M$. The latter will be identified with the vector $p\in\mathbb{R}^{n}$ through $p(\dot{x})=\langle p,\dot{x}\rangle$. Here $\langle\cdot,\cdot\rangle$ is the Euclidean scalar product on $\mathbb{R}^{n}$. For every $x\in M$, the fibers $T_{x}M$ and $T^{\ast}_{x}M$ are also endowed with the Euclidean norm.

\vspace{0.5cm}

Let $C(M,\mathbb{R}^{m})$ denote the Banach space of continuous vector-valued functions $\mathbf{u}$ from $M$ to $\mathbb{R}^{m}$, endowed with the norm $\|\mathbf{u}\|_{\infty,M}:=\max_{x\in M}\|\mathbf{u}(x)\|$. Similarly, given $T>0$, $C(M\times[0,T],\mathbb{R}^{m})$ denotes the Banach space of continuous vector-valued functions $\mathbf{u}$ from $M\times[0,T]$ to $\mathbb{R}^{m}$, endowed with the norm $\|\mathbf{u}\|_{\infty,M\times[0,T]}=\max_{(x,t)\in M\times[0,T]}\|\mathbf{u}(x,t)\|$. Now given $\kappa>0$, we say $\mathbf{u}\in C(M,\mathbb{R}^{m})$ is $\kappa$-Lipschitz if for any $x,y\in M$
$$
\|\mathbf{u}(x)-\mathbf{u}(y)\|\leq\kappa|x-y|,
$$
and by the definition of $\|\cdot\|$, it is equivalent to that each component of $\mathbf{u}$ is $\kappa$-Lipschitz. Similar definition applies for functions in $C(M\times[0,\infty),\mathbb{R}^{m})$, where a function is called Lipschitz only if it is Lipschitz continuous in both variables. Let $\Omega$ be a compact domain, $j\in\mathbb{N}$ and $\mathbf{u}\in C^{j}(\Omega,\mathbb{R}^{m})$, we use
$$
\|\mathbf{u}\|_{C^{j},\Omega}=\max_{|\alpha|\leq j}\{\|\partial^{\alpha}\mathbf{u}\|_{\infty,\Omega}\}
$$
to denote the $C^{j}$-norm of $\mathbf{u}$.

\vspace{0.5cm}

Let $O(h)$ ($o(h)$) denote some real-valued function $f$ defined on a neighborhood of $0$ such that $|\frac{f(h)}{h}|$ is bounded (goes to $0$) as $h$ goes to $0$ respectively.

\subsection{Assumptions}
Now we introduce assumptions on the Hamiltonians arose in system \eqref{HJs}. For $1\leq i\leq m$, let $H_{i}\in C^{2}(T^{\ast}M\times\mathbb{R}^{m},\mathbb{R})$ be
\begin{enumerate}[\bf (H1)]
  \item strictly convex with respect to $p$,

  \item superlinear with respect to $p$, i.e. for any $(x,\mathbf{u})$,
        $$
        \frac{H_{i}(x,p,\mathbf{u})}{|p|}\rightarrow\infty\hspace{0.2cm}\text{as}\hspace{0.2cm}|p|\rightarrow\infty,
        $$

  \item uniformly Lipschitz with respect to $\mathbf{u}$, i.e. there exists $\Theta>0$ such that
        $$
        |H_{i}(x,p,\mathbf{u})-H_{i}(x,p,\mathbf{v})|\leq\Theta\|\mathbf{u}-\mathbf{v}\|\quad\text{for all }x,p,\mathbf{u},\mathbf{v}.
        $$
\end{enumerate}

For $1\leq i\leq m$, let $L_{i}:TM\times\mathbb{R}^{m}\rightarrow\mathbb{R}$ be the convex dual associated to $H_{i}$, that is,
\begin{equation*}
L_{i}(x,\dot{x},\mathbf{u}):=\sup_{p\in T^{\ast}_{x}M}\{\langle p,\dot{x}\rangle-H(x,p,\mathbf{u})\}.
\end{equation*}

Then (H1)-(H3) are easily translated to the assumptions on $L_{i}\in C^{2}(TM\times\mathbb{R}^{m},\mathbb{R})$:
\begin{enumerate}[\bf (L1)]
  \item strictly convex with respect to $v$,

  \item superlinear with respect to $\dot{x}$, i.e. fix any $(x,\mathbf{u})$,
        $$
        \frac{L_{i}(x,\dot{x},\mathbf{u})}{|\dot{x}|}\rightarrow\infty\hspace{0.2cm}\text{as}\hspace{0.2cm}|\dot{x}|\rightarrow\infty,
        $$

  \item uniformly Lipschitz with respect to $\mathbf{u}$, i.e. there exists $\Theta>0$ such that
        $$
        |L_{i}(x,\dot{x},\mathbf{u})-L_{i}(x,\dot{x},\mathbf{v})|\leq\Theta\|\mathbf{u}-\mathbf{v}\|\quad\text{for all }x,\dot{x},\mathbf{u},\mathbf{v}.
        $$
\end{enumerate}

\begin{remark}
Technically, assumption (H2) also reads as: for any $C>0$ and compact set $\mathcal{K}\subset\mathbb{R}^{m}$, there exist positive constants $R(C,\mathcal{K}), D(C,\mathcal{K})$ such that for $\mathbf{u}\in\mathcal{K}$ and $p\in T^{\ast}_{x}M$ with $|p|\geq R$, $H(x,p,\mathbf{u})\geq C|p|-D$. By the dual property of $L$ and $H$, the above proposition is also valid for $L$ with $p$ replaced by $\dot{x}$.
\end{remark}

We add a brief review on generality of (H1)-(H3) or (L1)-(L3) here.

1. Assumptions (H1)-(H2) constitute the famous Tonelli conditions, which are very general conditions for proving the existence of the action minimizers in the theory of calculus of variation. To develop a global variational methods for positive definite Hamiltonian systems with arbitrary degree of freedom, J. Mather proposed (H1)-(H2) as basic assumptions in his celebrated papers \cite{M1},\cite{M2}.

Mather initially deal with the case $H=H(x,p,t)$ also depending periodically on time variable $t$, but contains no $\mathbf{u}$-variable. He further assume that the associated Hamilton flow is complete since there are examples whose action minimizers are not solutions to the corresponding Euler-Lagrange equation. In the case that $H=H(x,p)$ does not depend on $t$, (H1)-(H2) implies the completeness of the Hamilton flow.

2. If $m=1$, such a Hamiltonian $H=H(x,p,u)$ is called contact Hamiltonian. Recent study shows that (H1)-(H3) is also a suitable setting for generalizing the global variational methods including Aubry-Mather theory and weak KAM theory to contact Hamiltonian systems. For such a topic, we refer to the series of works by K. Wang, the second and third author, see \cite{SWY}, \cite{WWY1},\cite{WWY2}. The condition (H3) is crucial in establishing the implicit variational principle in the contact case.

\subsection{Definition of viscosity solution}
Since a solution to the weakly coupled system \eqref{HJs} is a vector-valued continuous function, the notion of viscosity solution should be reformulated. The following definition, given in \cite{CLL}, is a suitable candidate and can be viewed as a componentwise generalization of the classical definition.

\begin{definition}\label{vis}
Let $\mathbf{u}:M\times[0,\infty)\rightarrow\mathbb{R}^{m}$ be a continuous function,
\begin{enumerate}[(i)]
  \item it is called a viscosity subsolution of \eqref{HJs} if for each $1\leq i\leq m$,
        \begin{itemize}
          \item $u_{i}(\cdot,0)\leq\varphi_{i}$ on $M$,

          \item whenever $\phi$ is a real-valued $C^{1}$ function on a neighborhood of $(x,t),t>0$ such that $u_{i}-\phi$ attains a local maximum at $(x,t)$,
               \begin{equation*}
               \partial_{t}\phi(x,t)+H_{i}(x,\partial_{x}\phi(x,t),\mathbf{u}(x,t))\leq0.
               \end{equation*}
        \end{itemize}

  \item it is called a viscosity supersolution of \eqref{HJs} if for each $1\leq i\leq m$,
        \begin{itemize}
          \item $u_{i}(\cdot,0)\geq\varphi_{i}$ on $M$,

          \item whenever $\phi$ is a real-valued $C^{1}$ function on a neighborhood of $(x,t),t>0$ such that $u_{i}-\phi$ attains a local minimum at $(x,t)$,
               \begin{equation*}
               \partial_{t}\phi(x,t)+H_{i}(x,\partial_{x}\phi(x,t),\mathbf{u}(x,t))\geq0.
               \end{equation*}
        \end{itemize}

  \item it is called a viscosity solution if it is both a viscosity sub and supersolution of \eqref{HJs}.
\end{enumerate}
\end{definition}

From now on, solutions to the system \eqref{HJs} are always understood in the above sense.

\section{Variational solution and solution semigroup}
In this section, we construct a variational principle (dynamical programing principle) corresponding to system \eqref{HJs}, which is a multi-dimensional analogy of the one constructed for contact Hamilton-Jacobi equations, see \cite{SWY}-\cite{WWY2}.

Based on this variational principle, we define the notion of variational solution and solution semigroup associated to system \eqref{HJs}. Some necessary properties of the variational solution are proved through this procedure.

\vspace{0.5cm}
\begin{definition}
Fix $\varphi\in C(M,\mathbb{R}^{m})$, for a given $T>0$, we define an operator $\mathbb{A}_{\varphi}$ from $C(M\times[0,T],\mathbb{R}^{m})$ to itself. Let $\,\,\,\mathbf{u}\in C(\mathbb{T}^{n}\times[0,T],\mathbb{R}^{m})$,
\begin{equation}\label{eq:2}
\mathbb{A}_{\varphi}[\mathbf{u}]_{i}(x,t)=\inf_{\substack{\gamma_{i}(t)=x\\ \gamma_{i}\in C^{ac}([0,t],M)}}\bigg\{\varphi_{i}(\gamma_{i}(0))+\int^{t}_{0}L_{i}(\gamma_{i}(s),\dot{\gamma}_{i}(s),\mathbf{u}(\gamma_{i}(s),s))ds\bigg\}.
\end{equation}
\end{definition}

\begin{remark}
We note that, by Equation \eqref{eq:2}, $\mathbb{A}_{\varphi}$ does not depend on $T$; by Tonelli theorem (see for instance \cite{BGH}), for any $1\leq i\leq m$, the above infimum can be achieved.
\end{remark}

\begin{proposition}\label{fp}
For any $T>0$, $\mathbb{A}_{\varphi}$ admits a unique fixed point $\mathbf{u}_{\varphi,T}$ in $C(M\times[0,T],\mathbb{R}^{m})$.
\end{proposition}

\begin{proof}
Let $(x,t)\in M\times[0,T]$. For any $\mathbf{u}\in C(M\times[0,T],\mathbb{R}^{m})$ and $1\leq i\leq m$, let $\xi_{i,1}:[0,t]\rightarrow M$ be an absolutely continuous curve with $\xi_{i,1}(t)=x$ such that
\begin{equation*}
\mathbb{A}_{\varphi}[\mathbf{u}]_{i}(x,t) = \varphi_{i}(\xi_{i,1}(0))+ \int_{0}^{t} L_{i}(\xi_{i,1}(s),\dot{\xi}_{i,1}(s), \mathbf{u}(\xi_{i,1}(s),s))ds.
\end{equation*}

For any $\mathbf{v}\in C(M\times[0,T],\mathbb{R}^{m})$, from (L3) we have
\begin{align*}
&\mathbb{A}_{\varphi}[\mathbf{u}]_{i}(x,t)-\mathbb{A}_{\varphi}[\mathbf{v}]_{i}(x,t)\\ \leq&\int_{0}^{t}|L_{i}(\xi_{i,1}(s),\dot{\xi}_{i,1}(s), \mathbf{v}(\xi_{i,1}(s),s))- L_{i}(\xi_{i,1}(s),\dot{\xi}_{i,1}(s),\mathbf{u}(\xi_{i,1}(s),s))|ds\\
\leq&t\Theta\|\mathbf{u}-\mathbf{v}\|_{\infty,M\times[0,T]}.
\end{align*}

By exchanging the position of $\mathbf{u}$ and $\mathbf{v}$, we obtain

\begin{equation*}
|\mathbb{A}_{\varphi}[\mathbf{u}]_{i}(x,t)-\mathbb{A}_{\varphi}[\mathbf{v}]_{i}(x,t)|\leq t\Theta\|\mathbf{u}-\mathbf{v}\|_{\infty,M\times[0,T]},
\end{equation*}
which is equivalent to
\begin{equation}\label{p-1}
\|\mathbb{A}_{\varphi}[\mathbf{u}](x,t)-\mathbb{A}_{\varphi}[\mathbf{v}](x,t)\|\leq t\Theta\|\mathbf{u}-\mathbf{v}\|_{\infty,M\times[0,T]}.
\end{equation}

Let $\xi_{i,2}:[0,t]\rightarrow M$ be an absolutely continuous curve with $\xi_{i,2}(t)=x$ such that

\[
\mathbb{A}_{\varphi}^{2}[\mathbf{v}](x,t) = \varphi(\xi_{i,2}(0))+\int_{0}^{t}L_{i}(\xi_{i,2}(s), \dot{\xi}_{i,2}(s),\mathbb{A}_{\varphi} [\mathbf{v}](\xi_{i,2}(s),s)) ds.
\]
It follows from (\ref{p-1}) that  for $s\in [0,t]$, we have

\begin{equation*}
\|\mathbb{A}_{\varphi}[\mathbf{u}](\xi_{i,2}(s),s)-\mathbb{A}_{\varphi}[\mathbf{v}](\xi_{i,2}(s),s)\|\leq s\Theta\|\mathbf{u}-\mathbf{v}\|_{\infty,M\times[0,T]}.
\end{equation*}

Thus we have the following estimates

\begin{align*}
&\mathbb{A}_{\varphi}^{2}[\mathbf{v}]_{i}(x,t)-\mathbb{A}_{\varphi}^{2}[\mathbf{u}]_{i}(x,t)\\
\leq&\int_{0}^{t}\Theta\|\mathbb{A}_{\varphi}[\mathbf{u}](\xi_{i,2}(s),s)-
\mathbb{A}_{\varphi}[\mathbf{v}](\xi_{i,2}(s),s)\|ds\\
\leq&\Theta^{2}\|\mathbf{u}-\mathbf{v}\|_{\infty,M\times[0,T]}\int_{0}^{t}sds\\
\leq&\frac{(t\Theta)^2}{2}\|\mathbf{u}-\mathbf{v}\|_{\infty,M\times[0,T]}.
\end{align*}

By exchanging $\mathbf{u}$ and $\mathbf{v}$, we obtain
\begin{equation*}
\|\mathbb{A}_{\varphi}^{2}[\mathbf{u}](x,t)-\mathbb{A}_{\varphi}^{2}[\mathbf{v}](x,t)\|\leq \frac{(t\Theta)^{2}}{2}\|\mathbf{u}-\mathbf{v}\|_{\infty,M\times[0,T]}.
\end{equation*}

Continuing the above procedure, we obtain

\[
\|\mathbb{A}_{\varphi}^{n}[\mathbf{u}](x,t)-\mathbb{A}_{\varphi}^{n}[\mathbf{v}](x,t)\|
\leq\frac{(t\Theta)^{n}}{n!}\|\mathbf{u}-\mathbf{v}\|_{\infty,M\times[0,T]},
\]
which implies

\[
\|\mathbb{A}_{\varphi}^{n}[\mathbf{u}]-\mathbb{A}_{\varphi}^{n}[\mathbf{v}]\|_{\infty,M\times[0,T]}
\leq\frac{(T\Theta)^{n}}{n!}\|\mathbf{u}-\mathbf{v}\|_{\infty,M\times[0,T]}.
\]
Therefore, there exists  $\mathbf{N}\in \mathbb{N}$ large enough such that $\mathbb{A}_{\varphi}^{\mathbf{N}}$ is a
contraction mapping. Since $C(M\times[0,T],\mathbb{R}^{m})$ is complete, by Banach fixed point theorem, there exists a unique $\mathbf{u}_{\varphi,T}\in C(M\times[0,T], \mathbb{R}^{m})$ such that

\begin{equation*}
\mathbb{A}_{\varphi}^{\mathbf{N}}[\mathbf{u}_{\varphi,T}]=\mathbf{u}_{\varphi,T}.
\end{equation*}

Since
$$
\mathbb{A}_{\varphi}[\mathbf{u}_{\varphi,T}]=\mathbb{A}_{\varphi}\circ \mathbb{A}_{\varphi}^{\mathbf{N}}[\mathbf{u}_{\varphi,T}] = \mathbb{A}_{\varphi}^{\mathbf{N}}\circ\mathbb{A}_{\varphi}[\mathbf{u}_{\varphi,T}],
$$
$\mathbb{A}_{\varphi}[\mathbf{u}_{\varphi,T}]$ is also a fixed point of $\mathbb{A}_{\varphi}^{\mathbf{N}}$. By the uniqueness of fixed point of $\mathbb{A}_{\varphi}^{\mathbf{N}}$, we have

\[
\mathbb{A}_{\varphi}[\mathbf{u}_{\varphi,T}]=\mathbf{u}_{\varphi,T}.
\]
This completes the proof of Lemma \ref{fp}.
\end{proof}

\vspace{0.5cm}
For any $0<T<T^{\prime}$, denote by $\mathbf{u}_{\varphi,T^{\prime}}$ the fixed point of $\mathbb{A}_{\varphi}$ in $C(M\times[0,T^{\prime}])$. From the definition of $\mathbb{A}_{\varphi}$, $\mathbf{u}_{\varphi,T^{\prime}}|_{M\times[0,T]}$ is also a fixed point of $\mathbb{A}_{\varphi}$, thus by uniqueness of fixed point of $\mathbb{A}_{\varphi}$, $\mathbf{u}_{\varphi,T^{\prime}}|_{M\times[0,T]}=\mathbf{u}_{\varphi,T}$. Thus we define
\begin{equation}
\mathbf{u}(x,t)=\mathbf{u}_{\varphi,T}(x,t)\hspace{0.3cm}\text{ for }t\leq T.
\end{equation}
This $\mathbf{u}(x,t)$ coincides with the notion of variational solution associated to \eqref{HJs} defined as

\begin{definition}\label{vas}
Let $\mathbf{u}:M\times[0,\infty)\rightarrow\mathbb{R}$ be a continuous function, $\mathbf{u}$ is called a variational solution to the system \eqref{HJs} if for any $(x,t)\in M\times[0,\infty)$ and any $1\leq i\leq m$,
\begin{equation}\label{vas eq}
u_{i}(x,t)=\inf_{\substack{\gamma_{i}(t)=x\\ \gamma_{i}\in C^{ac}([0,t],M)}}\bigg\{\varphi_{i}(\gamma_{i}(0))+\int^{t}_{0}L_{i}(\gamma_{i}(s),\dot{\gamma}_{i}(s),\mathbf{u}(\gamma_{i}(s),s))ds\bigg\}.
\end{equation}
\end{definition}

\vspace{0.5cm}

The above definition disintegrate into the following dominated-calibrated properties.
\begin{proposition}\label{ca}
Let $\mathbf{u}$ be a variational solution associated to system \eqref{HJs}, then
\begin{enumerate}[\rm (i)]
  \item for any $1\leq i\leq m, 0\leq t_{1}<t_{2}$ and any absolutely continuous curve $\gamma_{i}:[t_{1},t_{2}]\rightarrow M$,
        \begin{equation}\label{cali1}
        u_{i}(\gamma_{i}(t_{2}),t_{2})-u_{i}(\gamma_{i}(t_{1}),t_{1})\leq\int^{t_{2}}_{t_{1}}L_{i}(\gamma_{i}(s),\dot{\gamma}_{i}(s),\mathbf{u}(\gamma_{i}(s),s))ds.
        \end{equation}

  \item for each $1\leq i\leq m$ and any $0\leq t_{1}<t_{2}, x\in M$, there exists an absolutely continuous curve $\xi_{i}:[t_{1},t_{2}]\rightarrow M$ with $\xi_{i}(t_{2})=x$ such that
        \begin{equation}
        u_{i}(\xi_{i}(t_{2}),t_{2})-u_{i}(\xi_{i}(t_{1}),t_{1})=\int^{t_{2}}_{t_{1}}L_{i}(\xi_{i}(s),\dot{\xi}_{i}(s),\mathbf{u}(\xi_{i}(s),s))ds.
        \end{equation}
\end{enumerate}
\end{proposition}

\begin{proof}
(i) In fact, there exists an absolutely continuous curve $\eta_{i}:[t_{1},t_{2}]\rightarrow M$ such that
\begin{equation*}
u_{i}(\eta_{i}(t_{2}),t_{2})-u_{i}(\eta_{i}(t_{1}),t_{1})>\int^{t_{2}}_{t_{1}}L_{i}(\eta_{i}(s),\dot{\eta}_{i}(s),\mathbf{u}(\eta_{i}(s),s))ds,
\end{equation*}
let $\xi_{i}\in C^{ac}([0,t_{1}],M)$ be a curve with $\xi_{i}(t_{1})=\eta_{i}(t_{1})$ such that
\begin{equation*}
u_{i}(\xi_{i}(t_{1}),t_{1})=\varphi_{i}(\xi_{i}(0))+\int^{t_{1}}_{0}L_{i}(\xi_{i}(s),\dot{\xi}_{i}(s),\mathbf{u}(\xi_{i}(s),s))ds,
\end{equation*}
then $\gamma_{i}:=\eta_{i}\star\xi_{i}\in C^{ac}([0,t_{2}],M)$ and
\begin{align*}
&u_{i}(\gamma_{i}(t_{2}),t_{2})-\varphi_{i}(\gamma_{i}(0))=u_{i}(\eta_{i}(t_{2}),t_{2})-\varphi_{i}(\gamma_{i}(0))\\
=&[u_{i}(\eta_{i}(t_{2}),t_{2})-u_{i}(\eta_{i}(t_{1}),t_{1})]-[u_{i}(\xi_{i}(t_{1}),t_{1})-\varphi_{i}(\xi_{i}(0))]\\
>&\int^{t_{2}}_{t_{1}}L_{i}(\eta_{i}(s),\dot{\eta}_{i}(s),\mathbf{u}(\eta_{i}(s),s))ds+\int^{t_{1}}_{0}L_{i}(\xi_{i}(s),\dot{\xi}_{i}(s),\mathbf{u}(\xi_{i}(s),s))ds\\
=&\int^{t_{2}}_{0}L_{i}(\gamma_{i}(s),\dot{\gamma}_{i}(s),\mathbf{u}(\gamma_{i}(s),s))ds,
\end{align*}
which contradicts to Definition \ref{vas}.

\vspace{0.5cm}

(ii) Let $\bar{\xi}_{i}\in C^{ac}([0,t_{2}],M)$ be a curve with $\bar{\xi}_{i}(t_{2})=x$ such that
\begin{equation}\label{cali}
u_{i}(\bar{\xi}_{i}(t_{2}),t_{2})=\varphi_{i}(\bar{\xi}_{i}(0))+\int^{t_{2}}_{0}L_{i}(\bar{\xi}_{i}(s),\dot{\bar{\xi}}_{i}(s),\mathbf{u}(\bar{\xi}_{i}(s),s))ds.
\end{equation}

Now by (i), for any $0\leq t_{1}<t_{2}$,
\begin{equation*}
\begin{split}
u_{i}(\bar{\xi}_{i}(t_{2}),t_{2})-u_{i}(\bar{\xi}_{i}(t_{1}),t_{1})\leq\int^{t_{2}}_{t_{1}}L_{i}(\bar{\xi}_{i}(s),\dot{\bar{\xi}}_{i}(s),\mathbf{u}(\bar{\xi}_{i}(s),s))ds,\\
u_{i}(\bar{\xi}_{i}(t_{1}),t_{1})-\varphi_{i}(\bar{\xi}_{i}(0),0)\leq\int^{t_{1}}_{0}L_{i}(\bar{\xi}_{i}(s),\dot{\bar{\xi}}_{i}(s),\mathbf{u}(\bar{\xi}_{i}(s),s))ds.
\end{split}
\end{equation*}
We add the above two inequalities and use \eqref{cali} to find that the inequalities are actually equalities. So we define $\xi_{i}=\bar{\xi}_{i}|_{[t_{1},t_{2}]}$ to complete the proof.
\end{proof}

Now let us define a family of operators $\{T^{-}_{t}\}_{t\in\mathbb{R}}$ from $C(M,\mathbb{R}^{m})$ to itself.

\begin{definition}\label{semigroup}
For each $\varphi\in C(M,\mathbb{R}^{m})$, let $\mathbf{u}\in C(M\times[0,\infty),\mathbb{R})$ be the variational solution associated to system \eqref{HJs}, define
\begin{equation}
T^{-}_{t}\varphi(x)=\mathbf{u}(x,t),\hspace{1cm}\forall(x,t)\in M\times[0,\infty)
\end{equation}
\end{definition}

Due to the following proposition, the operator family $\{T^{-}_{t}\}_{t\geq0}$ really constitutes a semigroup, we call such an operator family the solution semigroup associated to system \eqref{HJs}.

\begin{proposition}\label{sg}
For any $t,s\geq0$, $T^{-}_{t+s}=T^{-}_{t}\circ T^{-}_{s}$.
\end{proposition}

\begin{proof}
For every fixed $s\geq0$, we define
\begin{equation}
\begin{split}
\mathbf{u}(x,t)=T^{-}_{t}\circ T^{-}_{s}\varphi(x)\\
\mathbf{v}(x,t)=T^{-}_{t+s}\varphi(x).
\end{split}
\end{equation}

On one hand, by definition of $T^{-}_{t}$ and $\mathbf{u}$, for each $1\leq i\leq m$ and $0<t\leq T$,
\begin{align*}
&u_{i}(x,t)=\bigg[T^{-}_{t}\circ T^{-}_{s}\varphi\bigg]_{i}(x)\\
=&\inf_{\substack{\gamma_{i}(t)=x}}\bigg\{\bigg[T^{-}_{s}\varphi\bigg]_{i}(\gamma_{i}(0))+\int^{t}_{0}L_{i}(\gamma_{i}(\tau),\dot{\gamma}_{i}(\tau),T^{-}_{\tau}\circ T^{-}_{s}\varphi(\gamma_{i}(\tau)))d\tau\bigg\}\\
=&\inf_{\substack{\gamma_{i}(t)=x}}\bigg\{\bigg[T^{-}_{s}\varphi\bigg]_{i}(\gamma_{i}(0))+\int^{t}_{0}L_{i}(\gamma_{i}(\tau),\dot{\gamma}_{i}(\tau),\mathbf{u}(\gamma_{i}(\tau),\tau))d\tau\bigg\}\\
=&\mathbb{A}_{T^{-}_{s}\varphi}[\mathbf{u}]_{i}(x,t),
\end{align*}
where $\gamma_{i}\in C^{ac}([0,t],M)$, so
$$
\mathbf{u}=\mathbb{A}_{T^{-}_{s}\varphi}[\mathbf{u}].
$$

On the other hand, by definition of $T^{-}_{t}$ and $\mathbf{v}$, for each $1\leq i\leq m$ and $0<t\leq T$,
\begin{align*}
&v_{i}(x,t)=\bigg[T^{-}_{t+s}\varphi\bigg]_{i}(x)\\
=&\inf_{\substack{\gamma_{i}(t+s)=x}}\bigg\{\varphi_{i}(\gamma_{i}(0))+\int^{t+s}_{0}L_{i}(\gamma_{i}(\tau),\dot{\gamma}_{i}(\tau),T^{-}_{\tau}\varphi(\gamma_{i}(\tau)))d\tau\bigg\}\\
=&\inf_{\substack{\gamma_{i}(t+s)=x}}\bigg\{\varphi_{i}(\gamma_{i}(0))+\bigg(\int^{s}_{0}+\int^{t+s}_{s}\bigg)L_{i}(\gamma_{i}(\tau),\dot{\gamma}_{i}(\tau),T^{-}_{\tau}\varphi(\gamma_{i}(\tau)))d\tau\bigg\}\\
=&\inf_{\substack{\bar{\gamma}_{i}(t)=x}}\bigg\{\bigg[T^{-}_{s}\varphi\bigg]_{i}(\bar{\gamma}_{i}(0))+\int^{t}_{0}L_{i}(\bar{\gamma}_{i}(\tau),\dot{\bar{\gamma}}_{i}(\tau),T^{-}_{\tau+s}\varphi(\bar{\gamma}_{i}(\tau)))d\tau\bigg\}\\
=&\inf_{\substack{\bar{\gamma}_{i}(t)=x}}\bigg\{\bigg[T^{-}_{s}\varphi\bigg]_{i}(\bar{\gamma}_{i}(0))+\int^{t}_{0}L_{i}(\bar{\gamma}_{i}(\tau),\dot{\bar{\gamma}}_{i}(\tau),\mathbf{v}(\bar{\gamma}_{i}(\tau),\tau))d\tau\bigg\}\\
=&\mathbb{A}_{T^{-}_{s}\varphi}[\mathbf{v}]_{i}(x),
\end{align*}
where $\gamma_{i}\in C^{ac}([0,t+s],M)$ and $\bar{\gamma}_{i}\in C^{ac}([0,t],M)$, so
$$
\mathbf{v}=\mathbb{A}_{T^{-}_{s}\varphi}[\mathbf{v}].
$$

Since both $\mathbf{u},\mathbf{v}$ are fixed points of $\mathbb{A}_{T^{-}_{s}\varphi}$ in $C(M\times[0,T],\mathbb{R}^{m})$, they must be equal, i.e.,
$$
T^{-}_{t}\circ T^{-}_{s}\varphi(x)=\mathbf{u}(x,t)=\mathbf{v}(x,t)=T^{-}_{t+s}\varphi(x),\hspace{0.3cm}(x,t)\in M\times[0,\infty),
$$
this completes the proof.
\end{proof}

\section{Representation of the viscosity solution}
This section devotes to a proof of Theorem \ref{the1}. The main theme focus on the uniqueness of viscosity solution in our setting. Note that, to the best of our knowledge, the existence and uniqueness of viscosity solutions for system \eqref{HJs} was firstly established by \cite{EL,IK} under certain  monotonicity assumptions at early 90s.

We remark here that our method presents a more dynamical flavor and we do not pursue optimality of assumptions, namely (H1)-(H3), under which the Theorem holds true. As suggested by Prof. Ishii, for the existence and uniqueness of the solution to system \eqref{HJs} satisfying (H3), more general assumptions are enough.

\vspace{0.3cm}
\subsection{Representation of solutions to special scalar equations}
~\\
Our proof of Theorem \ref{the1} is based on the representation of viscosity solutions to scalar Hamilton-Jacobi equations. We begin to consider $H:T^{\ast}M\times\mathbb{R}^{m}\rightarrow\mathbb{R}$ satisfying (H1)-(H3) and the following Cauchy problem:
\begin{equation}\label{HJe}
\left\{
  \begin{array}{ll}
    \partial_{t}u+H(x,\partial_{x}u,\bar{\mathbf{u}}(x,t))=0, & \hbox{$(x,t)\in M\times[0,\infty)$;}\\
    u(0,x)=\varphi(x), & \hbox{$x\in M$.}
  \end{array}
\right.
\end{equation}
where $\bar{\mathbf{u}}\in C^{2}(M\times[0,\infty),\mathbb{R}^{m})$ and the initial data $\varphi\in C(M,\mathbb{R})$.

\vspace{0.3cm}
Equation \eqref{HJe} is a classical Hamilton-Jacobi equation with a Tonelli Hamiltonian
$$
h(x,p,t):=H(x,\partial_{x}u,\bar{\mathbf{u}}(x,t))
$$
Denote $\Phi^{t}_{h}$ the local Hamiltonian flow induced by $h$. First, it is easily seen that

\begin{lemma}\label{complete}
$\Phi^{t}_{h}$ is defined on $[0,\infty)$.
\end{lemma}

\begin{proof}
Note that, for any $0\leq t\leq T$, along $\Phi^{t}_{\tilde{H}}$,
\begin{equation*}
\left|\frac{d}{dt}h\circ\phi^{t}_{h}\right|=\left|\frac{\partial}{\partial t}h\circ\phi^{t}_{h}\right|\leq\Theta\|\partial_{t}\bar{\mathbf{u}}(x(t),t)\|\leq\Theta\|\bar{\mathbf{u}}\|_{C^{2},M\times[0,T]}.
\end{equation*}
Thus, the compactness of $M\times[0,T]$ and (H2) implies the lemma.
\end{proof}

\vspace{0.2cm}
Set
$$
l(x,\dot{x},t):=L(x,\dot{x},\bar{\mathbf{u}}(x,t)),
$$
where $L$ is the convex dual of $H$, then $l$ is the corresponding Lagrangian of $h$. Denote by $\phi^{t}_{l}$ the local Lagrangian flow induced by $l$. $\Phi^{t}_{h}$ and $\phi^{t}_{l}$ are related by $\mathcal{L}:TM\times[0,T]\rightarrow T^{\ast}M\times[0,T]:$
$$
(x,\dot{x},t)\mapsto(x,\partial_{\dot{x}}l(x,\dot{x},t),t),
$$
By Lemma \ref{complete} and the assumptions on $h$, for $t\geq0$,
$$
\Phi^{t}_{h}\circ\mathcal{L}=\mathcal{L}\circ\phi^{t}_{l}.
$$
In particular,
\begin{corollary}\label{complete'}
$\phi^{t}_{l}$ is also defined on $[0,\infty)$.
\end{corollary}

\vspace{0.3cm}
Define
\begin{equation}\label{rep}
u(x,t)=\inf_{\substack{\gamma(t)=x\\ \gamma\in C^{ac}([0,t],M)}}\bigg\{\varphi(\gamma(0))+\int^{t}_{0}l(\gamma(s),\dot{\gamma}(s),s)ds\bigg\},
\end{equation}
The above discussion and \cite[Theorem 3.7.2, Page 108]{Fa} implies that every absolutely continuous curve $\xi$ that achieves the infimum in \eqref{rep} is a $C^{2}$ orbit of $\phi^{t}_{l}$. This leads to the fact that $u$ is a viscosity solution to the Cauchy problem \eqref{HJe}. With a bit more work, we can show that

\begin{lemma}\label{Lip2}
$u$ is locally Lipschitzian on $M\times(0,\infty)$.
\end{lemma}

\begin{proof}
By Corollary \ref{complete'}, \cite[Proposition 4.4.4, Page 138]{Fa} states that, for every $t>0$, there is a constant $A_{t}>0$ such that for any $\xi:[0,t]\rightarrow M$ achieving the infimum in \eqref{rep},
\begin{equation*}
\int^{t}_{0}l(\xi(s),\dot{\xi}(s),s)ds\leq A_{t}.
\end{equation*}
Thus there is  $s_{0}\in[0,t]$ such that $l(\xi(s_{0}),\dot{\xi}(s_{0}),s_{0})=L(\xi(s_{0}),\dot{\xi}(s_{0}),\bar{\mathbf{u}}(\xi(s_{0}),s_{0}))\leq\frac{A_{t}}{t}$ and
\begin{equation*}
(\xi(s_{0}),\dot{\xi}(s_{0}))\in\mathcal{K}_{t}:=\{(x,\dot{x})|L(x,\dot{x},\mathbf{0})\leq\frac{A_{t}}{t}+\Theta\|\mathbf{u}\|_{C^{2},M\times[0,t]}\},
\end{equation*}
which is a compact subset of $TM$. For each $s\in[0,t]$, let $p(s)=\partial_{\dot{x}}l(\xi(s),\dot{\xi}(s),s)$, as in the proof of Lemma \ref{complete},
$$
H(\xi(s),p(s),\mathbf{0})\leq H_{t}+\Theta\|\mathbf{u}\|_{C^{2},M\times[0,t]}\cdot t
$$
where $H_{t}=\sup_{\mathcal{L}\mathcal{K}_{t}\times[0,t]}h(x,p,s)$. This implies that any minimizer $\xi:[0,t]\rightarrow M$ for the Equation \eqref{rep} is uniformly Lipschitzian with a Lipschitz constant only depend on $t$. Then the same proof of \cite[Lemma 4.6.3, Page 148]{Fa} shows the lemma.
\end{proof}

Now we recall a comparison result without proof, it is just an adaption of \cite[Corollary 5.1, Page 66]{B} to our case. For readers that are interested in the PDE aspects of the theory of viscosity solution, we recommend \cite{B} as an excellent survey.

\begin{lemma}\label{comparison}
Let $T>0$ and $h\in C(T^{\ast}M\times\mathbb{R},\mathbb{R})$, if $u,v\in C(M\times[0,T],\mathbb{R})$ are respectively sub and supersolution of
\begin{equation*}
\partial_{t}u+h(x,\partial_{x}u,t)=0,\hspace{0.3cm}\text{ for }(x,t)\in M\times(0,T)
\end{equation*}
and either $u$ or $v$ is Lipschitz continuous in $x$, uniformly with respect to $t$, then
\begin{equation*}
\max_{\substack{M\times[0,T]}}(u-v)^+\leq\max_{\substack{M}}(u(x,0)-v(x,0))^+,
\end{equation*}
where $(\cdot)^+:=\max\{\cdot,0\}$.
\end{lemma}

Thanks to Lemma \ref{Lip2}, we are able to show
\begin{lemma}\label{unique}
$u$ is the unique viscosity solution to the Cauchy problem \eqref{HJe}.
\end{lemma}

\begin{proof}
For any $\delta>0$, $u$ is Lipschitz continuous in $x$, uniformly in $t$ on $M\times[\delta,T]$. Assuming that there is another viscosity solution $v\in C(M\times[0,T],\mathbb{R})$, by Lemma \ref{comparison}, we conclude that
\begin{equation}\label{1}
\max_{M\times[0,T]}(u-v)^+\leq\max_{M}(u(\cdot,\delta)-v(\cdot,\delta))^+.
\end{equation}
Let $\delta$ in \eqref{1} goes to $0$, the continuity of $u$ and $v$ implies that $v\equiv u$ on $M\times[0,T]$, this completes the proof.
\end{proof}

\begin{remark}
It is worth mentioning here that, in general, for a given $L(x,\dot{x},t)\in C^2(TM\times[0,T],\mathbb{R})$, the action minimizing curves may not be of class $C^2$ and fail to be an extremal of the corresponding Euler-Lagrangian equation, see \cite{BM}.
\end{remark}

\subsection{Variational solution and viscosity solution}
~\\
Let $\mathbf{u}\in C(M\times[0,\infty),\mathbb{R}^{m})$ be the variational solution associated to \eqref{HJs}. By Definition \ref{vis}, to prove that $\mathbf{u}$ is a viscosity solution to the system \eqref{HJs}, we only need to show that, for any $1\leq i\leq m$, $u_{i}$ is a viscosity solution to the equation
\begin{equation}\label{hj}
\left\{
  \begin{array}{ll}
    \partial_{t}u+H_{i}(x,\partial_{x}u,\mathbf{u}(x,t))=0, & \hbox{$(x,t)\in M\times[0,\infty)$;} \\
    u(x,0)=\varphi_{i}(x), & \hbox{$x\in M$.}
  \end{array}
\right.
\end{equation}
Now we fix $i$ and $T>0$, for any $\epsilon>0$, there exists $\mathbf{u}_{\epsilon}\in C^{2}(M\times[0,T],\mathbb{R}^{m})$ such that $\|\mathbf{u}_{\epsilon}-\mathbf{u}\|_{\infty,M\times[0,T]}<\epsilon$. We set $h,h_{\epsilon}:T^{\ast}M\times[0,T]\rightarrow\mathbb{R}$ as
\begin{equation}\label{aux}
\begin{split}
h(x,p,t)=H_{i}(x,p,\mathbf{u}(x,t))\\
h_{\epsilon}(x,p,t)=H_{i}(x,p,\mathbf{u}_{\epsilon}(x,t))
\end{split}
\end{equation}

Combining the above construction and the results obtained before, we can show that
\begin{theorem}
The variational solution associated to system \eqref{HJs} is a viscosity solution.
\end{theorem}

\begin{proof}
Denote by $u_{i}^{\epsilon}:M\times[0,T]\rightarrow\mathbb{R}$ be the viscosity solution of
\begin{equation*}
\left\{
  \begin{array}{ll}
    \partial_{t}u+h_{\epsilon}(x,\partial_{x}u,t)=0, & \hbox{$(x,t)\in M\times[0,T]$;} \\
    u(x,0)=\varphi_{i}(x), & \hbox{$x\in M$,}
  \end{array}
\right.
\end{equation*}
then by Lemma \ref{unique}, for $(x,t)\in M\times[0,T]$,
\begin{equation}\label{rep1}
u_{i}^{\epsilon}(x,t)=\inf_{\substack{\gamma_{i}(t)=x\\ \gamma_{i}\in C^{ac}([0,t],M)}}\bigg\{\varphi(\gamma_{i}(0))+\int^{t}_{0}L_{i}(\gamma_{i}(s),\dot{\gamma}_{i}(s),\mathbf{u}_{\epsilon}(\gamma_{i}(s),s))ds\bigg\}.
\end{equation}

By Definition \ref{vas} and \eqref{rep1}, for any $0\leq t\leq T$,
\begin{align*}
&|u_{i}(x,t)-u_{i}^{\epsilon}(x,t)|\\
\leq&\sup_{\gamma_{i}}\bigg|\int^{t}_{0}L_{i}(\gamma_{i}(s),\dot{\gamma}_{i}(s),\mathbf{u}(\gamma_{i}(s),s))ds-\int^{t}_{0}L_{i}(\gamma_{i}(s),\dot{\gamma}_{i}(s),\mathbf{u}_{\epsilon}(\gamma_{i}(s),s))ds\bigg|\\
\leq&\sup_{\gamma_{i}}\int^{t}_{0}|L_{i}(\gamma_{i}(s),\dot{\gamma}_{i}(s),\mathbf{u}(\gamma_{i}(s),s))-L_{i}(\gamma_{i}(s),\dot{\gamma}_{i}(s),\mathbf{u}_{\epsilon}(\gamma_{i}(s),s))|ds\\
\leq&\Theta\sup_{\gamma_{i}}\int^{t}_{0}\|\mathbf{u}(\gamma_{i}(s),s))-\mathbf{u}_{\epsilon}(\gamma_{i}(s),s)\|ds\\
\leq&\Theta t\|\mathbf{u}_{\epsilon}-\mathbf{u}\|_{\infty,M\times[0,T]}\leq\Theta T\|\mathbf{u}_{\epsilon}-\mathbf{u}\|_{\infty,M\times[0,T]}\leq\Theta T\epsilon,
\end{align*}
where $\gamma_{i}\in C^{ac}([0,t],M)$ with $\gamma_{i}(t)=x$. Thus $\|u_{i}-u_{i}^{\epsilon}\|_{\infty,M\times[0,T]}\leq\Theta T\epsilon$, which implies that $u_{i}^{\epsilon}$ converges uniformly to $u_{i}$ on $M\times[0,T]$ as $\epsilon\rightarrow0$. Now we have
\begin{itemize}
  \item $h_{\epsilon}(x,p,t)$ converges uniformly to $h(x,p,t)$ on $T^{\ast}M\times[0,T]$ as $\epsilon\rightarrow0$, more precisely,
        \begin{equation*}
        |h_{\epsilon}(x,p,t)-h(x,p,t)|\leq\Theta\|\mathbf{u}_{\epsilon}-\mathbf{u}\|_{\infty,M\times[0,T]}\leq\Theta\epsilon,
        \end{equation*}

  \item $u_{i}^{\epsilon}$ is the viscosity solution to $\partial_{t}u+h_{\epsilon}(x,\partial_{x}u,t)=0$ with $u_{i}^{\epsilon}(x,0)=\varphi_i(x)$,

  \item $u_{i}^{\epsilon}$ converges uniformly to $u_{i}$ on $M\times[0,T]$ as $\epsilon\rightarrow0$.
\end{itemize}
By the stability of viscosity solutions, $u_{i}$ is a viscosity solution to \eqref{hj} on $M\times[0,T]$. Since $T>0$ is arbitrary, we complete the proof.
\end{proof}

\vspace{0.3cm}
Finally, we are able to show that
\begin{theorem}
Viscosity solution to the system \eqref{HJs} is identical with the variational solution.
\end{theorem}

\begin{proof}
Let $\mathbf{v}$ be a viscosity solution to the system \eqref{HJs}, by Definition \ref{vis}, its components $v_{i}$ is a viscosity solution to the scalar Equation \eqref{hj} with $\mathbf{u}$ replaced by $\mathbf{v}$. We shall prove that for $1\leq i\leq m, v_{i}\equiv u_{i}$ on $M\times[0,\infty)$, where $u_{i}$ is the $i$-th component of $\mathbf{u}$.

Let $c_{\epsilon}:=\Theta\epsilon>0$, by using the notation introduced in \eqref{aux}, we have
\begin{equation}\label{compa eq}
h_{\epsilon}(x,p,t)-c_{\epsilon}<h(x,p,t)<h_{\epsilon}(x,p,t)+c_{\epsilon}
\end{equation}
and that $h_{\epsilon}(x,p,t)\pm c_{\epsilon}$ converges uniformly to $h(x,p,t)$ on $T^{\ast}M\times[0,T]$ as $\epsilon\rightarrow0$.

\vspace{0.3cm}
Let $u_{i,\pm}^{\epsilon}:M\times[0,T]\rightarrow\mathbb{R}$ be the unique viscosity solution to
\begin{equation*}
\left\{
  \begin{array}{ll}
    \partial_{t}u+h_{\epsilon}(x,\partial_{x}u,t)\pm c_{\epsilon}=0, & \hbox{$(x,t)\in M\times[0,T]$;} \\
    u(x,0)=\varphi_{i}(x), & \hbox{$x\in M$,}
  \end{array}
\right.
\end{equation*}
respectively. By Lemma \ref{Lip2}, $u_{i,\pm}^{\epsilon}$ is Lipschitz continuous in $x$, uniformly in $t$ on $M\times[\delta,T]$. On the other hand, by \eqref{compa eq}, $v_{i}|_{M\times(\delta,T)}$ is supersolution and subsolution to
\begin{equation*}
\partial_{t}u+h_{\epsilon}(x,\partial_{x}u,t)+c_{\epsilon}=0,\hspace{0.3cm}\text{ for }(x,t)\in M\times(\delta,T),
\end{equation*}
and
\begin{equation*}
\partial_{t}u+h_{\epsilon}(x,\partial_{x}u,t)-c_{\epsilon}=0,\hspace{0.3cm}\text{ for }(x,t)\in M\times(\delta,T),
\end{equation*}
respectively. Lemma \ref{comparison} implies that
\begin{equation}\label{compa eq2}
\begin{split}
\max_{\substack{M\times[\delta,T]}}(u_{i,+}^{\epsilon}-v_{i})^+\leq\max_{\substack{M}}(u_{i,+}^{\epsilon}(\cdot,\delta)-v_{i}(\cdot,\delta))^+\\
\max_{\substack{M\times[\delta,T]}}(v_{i}-u_{i,-}^{\epsilon})^+\leq\max_{\substack{M}}(v_{i}(\cdot,\delta)-u_{i,-}^{\epsilon}(\cdot,\delta))^+.
\end{split}
\end{equation}
By letting $\delta$ in \eqref{compa eq2} goes to $0$, we obtain
\begin{equation}\label{ord}
u^{\epsilon}_{i,+}(x,t)\leq v_{i}(x,t)\leq u^{\epsilon}_{i,-}(x,t),\hspace{0.3cm}\text{ for }(x,t)\in M\times[0,T].
\end{equation}

By Lemma \ref{unique},
\begin{equation*}
u^{\epsilon}_{i,\pm}(x,t)=\inf_{\substack{\gamma_{i}(t)=x\\ \gamma_{i}\in C^{ac}([0,t],M)}}\bigg\{\varphi(\gamma_{i}(0))+\int^{t}_{0}[L_{i}(\gamma_{i}(s),\dot{\gamma}_{i}(s),\mathbf{u}_{\epsilon}(\gamma_{i}(s),s))\mp c_{\epsilon}]ds\bigg\},
\end{equation*}
The above formula shows that $\|u_{i,+}^{\epsilon}-u_{i,-}^{\epsilon}\|_{\infty,M\times[0,T]}\leq2c_{\epsilon}T$. So by \eqref{ord},
$$
D^{1}_{\epsilon}:=\|v_{i}-u_{i,-}^{\epsilon}\|_{\infty,M\times[0,T]}\leq2c_{\epsilon}T.
$$
For $0\leq t\leq T$, set $D_{\epsilon}:=\|u_{i}-u_{i,-}^{\epsilon}\|_{\infty,M\times[0,T]}$, then
\begin{align*}
D^{2}_{\epsilon}\leq&\sup_{\gamma_{i}}\bigg|\int^{t}_{0}L_{i}(\gamma_{i}(s),\dot{\gamma}_{i}(s),\mathbf{u}(\gamma_{i}(s),s))ds-\int^{t}_{0}[L_{i}(\gamma_{i}(s),\dot{\gamma}_{i}(s),\mathbf{u}_{\epsilon}(\gamma_{i}(s),s))+c_{\epsilon}]ds\bigg|\\
\leq&\sup_{\gamma_{i}}\int^{t}_{0}|L_{i}(\gamma_{i}(s),\dot{\gamma}_{i}(s),\mathbf{u}(\gamma_{i}(s),s))-L_{i}(\gamma_{i}(s),\dot{\gamma}_{i}(s),\mathbf{u}_{\epsilon}(\gamma_{i}(s),s))|ds+c_{\epsilon}t\\
\leq&\sup_{\gamma_{i}}\int^{t}_{0}\Theta|\mathbf{u}(\gamma_{i}(s),s))-\mathbf{u}_{\epsilon}(\gamma_{i}(s),s)|ds+c_{\epsilon}t\leq2c_{\epsilon}T,
\end{align*}
where $\gamma_{i}\in C^{ac}([0,t],M)$ with $\gamma_{i}(t)=x$. So for $(x,t)\in M\times[0,T]$,
\begin{equation*}
|v_{i}(x,t)-u_{i}(x,t)|\leq|v_{i}(x,t)-u^{\epsilon}_{i,-}(x,t)|+|u_{i}(x,t)-u^{\epsilon}_{i,-}(x,t)|\leq D^{1}_{\epsilon}+D^{2}_{\epsilon}\leq4c_{\epsilon}T.
\end{equation*}
We complete the proof by letting $\epsilon$ goes to $0$.
\end{proof}

\section{Regularizing effect in a model problem}
In this section, we shall prove Theorem \ref{Thm Lip11} for system \eqref{HJs} with Hamiltonians satisfying (H1)-(H3) and (H*). Such model systems generalize the model building from classical mechanics with weakly coupled potential, thus are typical examples of system \eqref{HJs} of Tonelli type.

We begin to note that for $1\leq i\leq m, L_{i}:TM\times\mathbb{R}^{m}\rightarrow\mathbb{R}$ satisfies
\begin{itemize}[\bf (L*)]
  \item \begin{equation}\label{eq:3}
        L_{i}(x,\dot{x},\mathbf{u})=l_{i}(x,\dot{x})-P_{i}(x,\mathbf{u}),
        \end{equation}
        and
        \begin{align*}
        \frac{1}{A}|\dot{x}|^{r^{\ast}}\leq l_{i}(x,\dot{x})&\leq A(|\dot{x}|^{r^{\ast}}+1),\\
        |\partial_{x}l_{i}(x,\dot{x})|&\leq A(|\dot{x}|^{r^{\ast}}+1).
       \end{align*}
\end{itemize}
Moreover, for $|\dot{x}|\geq A$,
\begin{equation}\label{add}
|\partial_{\dot{x}}l_{i}(x,\dot{x})|\leq A|\dot{x}|^{r^{\ast}-1},
\end{equation}
where $r^{\ast}=\frac{r}{r-1}$ is the conjugate index with respect to $r$. We use the notation $\mathbf{l}:=(l_{1},...,l_{m}),\mathbf{P}:=(P_{1},...,P_{m})$ for later use.

\begin{remark}
We may choose $A>0$ sufficiently large to ensure both (H*) and (L*) hold. By (L3) and (L*), $\mathbf{P}$ is Lipschitzian in $\mathbf{u}$ with Lipschitz constant $\Theta$.
\end{remark}

\begin{remark}
We interpret the additional inequality \eqref{add} here. Set $p=\frac{\partial l_{i}}{\partial\dot{x}}(x,\dot{x})$, then
\begin{equation}\label{eq:4}
h_{i}(x,p)=\sup_{\dot{x}\in T_{x}M}\{\langle p,\dot{x}\rangle-l_{i}(x,\dot{x})\}=\langle\frac{\partial l_{i}}{\partial\dot{x}}(x,\dot{x}),\dot{x}\rangle-l_{i}(x,\dot{x}).
\end{equation}
By (L1), $l_{i}(x,\dot{x})-l_{i}(x,0)\leq\langle p,\dot{x}\rangle\leq l_{i}(x,2\dot{x})-l_{i}(x,\dot{x})$. Combining with \eqref{eq:4} and (L*),
\begin{equation}\label{eq:6}
\frac{1}{A}|\dot{x}|^{r^{\ast}}\leq\langle p,\dot{x}\rangle\leq2^{r^{\ast}}A|\dot{x}|^{r^{\ast}}+A,\quad h_{i}(x,p)\leq2^{r^{\ast}}A|\dot{x}|^{r^{\ast}}+A
\end{equation}
\vspace{0.2cm}
Notice that there is $\dot{x}^{\prime}\in T_{x}M$ such that $|\dot{x}^{\prime}|=|\dot{x}|$ and $\langle\dot{x}^{\prime}, p\rangle=|\dot{x}|\cdot|p|$, thus
\vspace{0.1cm}
$$
|\dot{x}|\cdot|p|-A(|\dot{x}|^{r^{\ast}}+1)\leq\langle\dot{x}^{\prime}, p\rangle-l(x,\dot{x}^{\prime})\leq h_{i}(x,p)\leq2^{r^{\ast}}A|\dot{x}|^{r^{\ast}}+A.
$$
Combining Equation \eqref{eq:6}, for $|\dot{x}|>A$,
\begin{equation}\label{eq:11}
\frac{1}{A}|\dot{x}|^{r^{\ast}-1}\leq|p|\leq(2^{r^{\ast}}+3)A|\dot{x}|^{r^{\ast}-1}.
\end{equation}
We could replace $A$ by $(2^{r^{\ast}}+3)A$ to ensure that both (L*) and Equation \eqref{add} hold.
\end{remark}

Let $t\in[0,T],\mathbf{u}\in C(M\times[0,T],\mathbb{R}^{m})$, we set
\begin{itemize}[$-$]
  \item $l(A):=\|\mathbf{l}\|_{C^{1},\{(x,\dot{x})||\dot{x}|\leq A\}}$.

  \item $U(t):=\sup_{n\geq0}\|\mathbb{A}^{n}_{\varphi}[\mathbf{u}]\|_{\infty}$, which is finite by Proposition \ref{fp}.

  \item $P(t):=\|\mathbf{P}\|_{C^{1},M\times[-U(t),U(t)]^{m}}$.

  \item $F(t):=2U(t)+t\cdot(P(t)+1)$.
\end{itemize}
where $\mathbb{A}_{\varphi}$ is defined by Equation \eqref{eq:2}. We are ready to show the following \textit{a priori} estimates:
\begin{lemma}\label{x-Lip lem}
Let
\begin{align*}
t_{\Theta}&:=\min\bigg\{1,\bigg(\frac{1+\frac{1}{r^{\ast}}}{2\Theta}\bigg)^{r^{\ast}}\bigg\}\\
\kappa_{0}&:=2(2C_{1}+C_{2})\max\{F(1),1\}+P(1),
\end{align*}
where $C_{i}>0, i=1,2$ are constants only depend on $A, r$. If $\mathbf{u}\in C(M\times[0,T],\mathbb{R}^{m})$ satisfies for any $x,y\in M$ and $t\in(0,t_{\Theta}]$,
\begin{equation}
\|\mathbf{u}(x,t)-\mathbf{u}(y,t)\|\leq\frac{\kappa_{0}}{t^{\frac{1}{r}}}\cdot|x-y|,
\end{equation}
then $\mathbb{A}_{\varphi}[\mathbf{u}]\in C(M\times[0,T],\mathbb{R}^{m})$ also satisfies for $t\in(0,t_{\Theta}]$,
\begin{equation}
\|\mathbb{A}_{\varphi}[\mathbf{u}](x,t)-\mathbb{A}_{\varphi}[\mathbf{u}](y,t)\|\leq\frac{\kappa_{0}}{t^{\frac{1}{r}}}\cdot|x-y|.
\end{equation}
\end{lemma}

\begin{proof}
Fix $(x,t)\in M\times[0,t_{\Theta}]$, let $\Delta:=y-x$. At the beginning stage, we assume $|\Delta|\leq t$. Let $\xi_{1}:[0,t]\rightarrow M$ be an absolute continuous curve with $\xi_{1}(t)=x$ such that
\begin{equation*}
\mathbb{A}_{\varphi}[\mathbf{u}]_{1}(x,t)=\varphi_{1}(\xi_{1}(0))+\int^{t}_{0}L_{1}(\xi_{1}(s),\dot{\xi}_{1}(s),\mathbf{u}(\xi_{1}(s),s))ds.
\end{equation*}
As usual, we define $\xi^{\Delta}_{1}:[0,t]\rightarrow M$ as
\begin{equation}\label{curve 0}
\xi^{\Delta}_{1}(s)=\xi_{1}(s)+\frac{s}{t}\Delta,\,\,\,s\in[0,t].
\end{equation}
From the above construction and the assumption $|\Delta|\leq t$, we have
\begin{enumerate}[(i)]
  \item $\xi^{\Delta}_{1}(0)=\xi_{1}(0)=x,\xi^{\Delta}_{1}(t)=\xi_{1}(t)+\Delta=y$,

  \item $|\xi^{\Delta}_{1}(s)-\xi_{1}(s)|=\frac{s}{t}|\Delta|\leq|\Delta|$,

  \item $|\dot{\xi}^{\Delta}_{1}(s)-\dot{\xi}_{1}(s)|=\frac{|\Delta|}{t}\leq1$.
\end{enumerate}

By (L*) and a direct calculation,
$$
\mathbb{A}_{\varphi}[\mathbf{u}]_1(y,t)-\mathbb{A}_{\varphi}[\mathbf{u}]_1(x,t)\leq I_{1}+I_{2}+I_{3},
$$
where
\begin{equation*}
\begin{split}
I_{1}=&\int^{t}_{0}|l_{1}(\xi^{\Delta}_{1}(s),\dot{\xi}^{\Delta}_{1}(s))-l_{1}(\xi^{\Delta}_{1}(s),\dot{\xi}_{1}(s))|ds,\\
I_{2}=&\int^{t}_{0}|l_{1}(\xi^{\Delta}_{1}(s),\dot{\xi}_{1}(s))-l_{1}(\xi_{1}(s),\dot{\xi}_{1}(s))|ds,\\
I_{3}=&\int^{t}_{0}|P_{1}(\xi^{\Delta}_{1}(s),\mathbf{u}(\xi^{\Delta}_{1}(s),s))-P_{1}(\xi_{1}(s),\mathbf{u}(\xi_{1}(s),s))|ds.
\end{split}
\end{equation*}

First, we observe that
\begin{equation*}
|\int_{0}^{t}l_{1}(\xi_{1}(s),\dot{\xi}_{1}(s))ds|=|\mathbb{A}_{\varphi}[\mathbf{u}]_{1}(x,t)-\varphi_{1}(\xi_{1}(0))+\int^{t}_{0}P_{1}(\xi_{1}(s),\mathbf{u}(\xi_{1}(s),s))ds|,\\
\end{equation*}
this leads to
\begin{equation}\label{eq:5}
|\int_{0}^{t}l_{1}(\xi_{1}(s),\dot{\xi}_{1}(s))ds|\leq 2U(t)+t\cdot P(t)\leq F(t).
\end{equation}

Now using \eqref{curve 0}, assumption (L*) and above inequalities, we estimate
\begin{align}\label{eq:8}
I_{1}\leq&\frac{|\Delta|}{t}\int^{t}_{0}|\partial_{\dot{x}}l_{1}(\xi^{\Delta}_{1}(s),\dot{\xi}_{1}(s)+\theta(s)\frac{\Delta}{t})|ds \nonumber\\
\leq&\frac{|\Delta|}{t}\int^{t}_{0}[A|\dot{\xi}_{1}(s)+\theta(s)\frac{\Delta}{t}|^{r^{\ast}-1}+l(A)]ds \nonumber\\
\leq&\frac{A|\Delta|}{t}\bigg[\int^{t}_{0}|\dot{\xi}_{1}(s)+\theta(s)\frac{\Delta}{t}|^{r^{\ast}}ds\bigg]^{\frac{1}{r}}\cdot\bigg[\int^{t}_{0}ds\bigg]^{\frac{1}{r^{\ast}}}+l(A)|\Delta| \nonumber\\
\leq&\frac{A|\Delta|}{t^{\frac{1}{r}}}\bigg[2^{r^{\ast}}\int^{t}_{0}|\dot{\xi}_{1}(s)|^{r^{\ast}}+|\theta(s)\frac{\Delta}{t}|^{r^{\ast}}ds\bigg]^{\frac{1}{r}}+l(A)|\Delta| \nonumber\\
\leq&\frac{A|\Delta|}{t^{\frac{1}{r}}}\bigg[2^{r^{\ast}}A\int^{t}_{0}[l_{1}(\xi_{1}(s),\dot{\xi}_{1}(s))+1]ds\bigg]^{\frac{1}{r}}+l(A)|\Delta| \nonumber\\
\leq&C_{1}(\frac{F(t)}{t}+1)^{\frac{1}{r}}|\Delta|.
\end{align}
where $C_{1}=2^{\frac{r^{\ast}}{r}}A^{1+\frac{1}{r}}+l(A)$. The first inequality uses mean value theorem; for the second inequality, we note that, by (L*) and the definition of $l(A)$, no matter $|\dot{x}|>A$ or $|\dot{x}|\leq A, l_{1}(x,\dot{x})\leq A|\dot{x}|^{r}+l(A)$; the third one is H\"{o}lder inequality; the fifth inequality also use (L*).

\vspace{0.3cm}
For the second part, we calculate as
\begin{align*}
I_{2}\leq&\int^{t}_{0}|\partial_{x}l_{1}(\xi_{1}(s)+\theta(s)\frac{s}{t}\Delta,\dot{\xi}_{1}(s))|\cdot \frac{s}{t}|\Delta|ds\\
\leq&\bigg[\int^{t}_{0}|\partial_{x}l_{1}(\xi_{1}(s)+\theta(s)\frac{s}{t}\Delta,\dot{\xi}_{1}(s))|ds\bigg]\cdot|\Delta|\\
\leq&\bigg[\int^{t}_{0}A(|\dot{\xi}_{1}(s)|^{r^{\ast}}+1)ds\bigg]\cdot|\Delta|\\
\leq&A^{2}|\Delta|\int^{t}_{0}l_{1}(\xi_{1}(s),\dot{\xi}_{1}(s))ds+At|\Delta|\leq C_{2}F(t)|\Delta|,
\end{align*}
where $C_{2}=A^{2}$. The first inequality uses mean value theorem; the second one uses $s\leq t$; the third and fourth one use (L*); the last inequality uses \eqref{eq:5}.

\vspace{0.3cm}
Now we get into the estimate of the last part\,:
\begin{align*}
I_{3}&\leq\Theta\int^{t}_{0}\|\mathbf{u}(\xi^{\Delta}_{1}(s),s)-\mathbf{u}(\xi_{1}(s),s)\|ds+P(t)\int^{t}_{0}|\xi^{\Delta}_{1}(s)-\xi_{1}(s)|ds.\\
\end{align*}

It is directly from Equation \eqref{curve 0} that
\begin{equation}
\int^{t}_{0}|\xi^{\Delta}_{1}(s)-\xi_{1}(s)|ds=\frac{|\Delta|}{t}\int^{t}_{0}sds=\frac{t}{2}|\Delta|.
\end{equation}

By the assumptions on $\mathbf{u}$, if $t\leq t_{\Theta}$, we have:
\begin{align}\label{eq:9}
&\int^{t}_{0}\|\mathbf{u}(\xi^{\Delta}_{1}(s),s)-\mathbf{u}(\xi_{1}(s),s)\|ds \nonumber\\
\leq&\int_{0}^{t}\kappa_{\varphi}(s)|\xi^{\Delta}_{1}(s)-\xi_{1}(s)|ds=\frac{\kappa_{0}|\Delta|}{t}\int_{0}^{t}s^{\frac{1}{r^{\ast}}}ds\leq\frac{\kappa_{0}}{1+\frac{1}{r^{\ast}}}t^{\frac{1}{r^{\ast}}}\cdot|\Delta|.
\end{align}

Combining \eqref{eq:8}-\eqref{eq:9}, for $t\leq t_{\Theta}$ and $x,y$ on $M$ with $|x-y|\leq t$, we have
\begin{align*}
&\mathbb{A}_{\varphi}[\mathbf{u}]_1(y,t)-\mathbb{A}_{\varphi}[\mathbf{u}]_1(x,t)\\
\leq&[C_{1}(\frac{F(1)}{t}+1)^{\frac{1}{r}}+C_{2}F(1)+\frac{P(1)}{2}+\frac{\Theta\kappa_{0}}{1+\frac{1}{r^{\ast}}}t^{\frac{1}{r^{\ast}}}]|\Delta|\\
\leq&[\frac{\kappa_{0}}{2t^{\frac{1}{r}}}+\frac{\kappa_{0}}{2}]\cdot|\Delta|\leq\frac{\kappa_{0}}{t^{\frac{1}{r}}}\cdot|\Delta|.
\end{align*}

Exchanging the role of $x$ and $y$, we have
\begin{equation}\label{Lip1}
|\mathbb{A}_{\varphi}[\mathbf{u}]_1(y,t)-\mathbb{A}_{\varphi}[\mathbf{u}]_1(x,t)|\leq\frac{\kappa_{0}}{t^{\frac{1}{r}}}|y-x|.
\end{equation}
We note that the estimation holds for $u_{2},...,u_{m}$ by the same procedure as above.

To complete the proof, we eliminate the restriction that $|\Delta|\leq t$: for a given $t\in(0,t_{\Theta}]$, let $x,y$ be any points on $M$, there exists $k\in\mathbb{N}$ and $x=x_{0},x_{1},...,x_{k}=y$ such that $|x_{j}-x_{j+1}|\leq t,j=1,...,k$ and $\sum_{j=0}^{k-1}|x_{j}-x_{j+1}|=|x-y|$. Then by \eqref{Lip1}, for each $i$,
\begin{align*}
|u_{i}(x,t)-u_{i}(y,t)|\leq&\sum_{j=0}^{k-1}|u_{i}(x_{j},t)-u_{i}(x_{j+1},t)|\\
\leq&\sum_{j=0}^{k-1}\|\mathbf{u}(x_{j},t)-\mathbf{u}(x_{j+1},t)\|\\
\leq&\frac{\kappa_{0}}{t^{\frac{1}{r}}}\sum_{j=0}^{k-1}|x_{j}-x_{j+1}|=\frac{\kappa_{0}}{t^{\frac{1}{r}}}|x-y|.
\end{align*}
\end{proof}

\textit{Proof of Theorem \ref{Thm Lip11}: }
Fix $\kappa_{0},t_{\Theta}$ as above, we construct $\mathcal{L}_{r}\subseteq C(M\times[0,t_{\Theta}])$ as
$$
\mathcal{L}_{r}:=\bigg\{\mathbf{u}\bigg|\|\mathbf{u}(y,t)-\mathbf{u}(x,t)\|\leq\frac{\kappa_{0}}{t^{\frac{1}{r}}}|y-x|\text{ for any }t>0,x,y\in M\bigg\},
$$
then $\mathcal{L}_{r}$ is complete under $\|\cdot\|_{\infty,M\times[0,t_{\Theta}]}$. By Lemma \ref{x-Lip lem}, $\mathbb{A}_{\varphi}:\mathcal{L}_{r}\rightarrow\mathcal{L}_{r}$. Take $\mathbf{u}_{0}\in\mathcal{L}_{r}$, by Lemma \ref{fp}, $\{\mathbb{A}^{n}_{\varphi}\mathbf{u}_{0}\}_{n\geq0}\subseteq\mathcal{L}_{r}$ is a Cauchy sequence in $C(M\times[0,t_{\Theta}])$. The unique limit of $\{\mathbb{A}^{n}_{\varphi}\mathbf{u}_{0}\}_{n\geq0}$ is the solution $T^{-}_{t}\varphi$ to \eqref{HJs} with $t\in[0,t_{\Theta}]$. Thus by the completeness of $\mathcal{L}_{r}, T^{-}_{t}\varphi\in\mathcal{L}_{r}$, i.e., for any $t\in(0,t_{\Theta}]$ and any two points $x,y$ on $M$,
\begin{equation}\label{eq:Lip}
\|T^{-}_{t}\varphi(x)-T^{-}_{t}\varphi(y)\|\leq\frac{\kappa_{0}}{t^{\frac{1}{r}}}|x-y|,
\end{equation}
where the constant $\kappa_{0}$ depends on the initial data $\varphi$. So for any $t\geq t_{\Theta}$, we use Proposition \ref{sg} and Equation \eqref{eq:Lip} to obtain that
$$
\|\mathbf{u}(x,t)-\mathbf{u}(y,t)\|=\|T^{-}_{t_{\Theta}}\mathbf{u}(x,t-t_{\Theta})-T^{-}_{t_{\Theta}}\mathbf{u}(y,t-t_{\Theta})\|\leq\kappa_{\varphi}(t)\cdot|x-y|,
$$
where
$$
\kappa_{\varphi}(t):=\kappa_{0}(F(1+t-t_{\Theta}),P(1+t-t_{\Theta}))\cdot t^{-\frac{1}{r}}_{\Theta}.
$$
It is immediate that $\kappa_{\varphi}(t)$ is a positive continuous function and is monotone increasing. In particular, $\kappa_{\varphi}(t)$ is locally bounded on $[t_{\Theta},\infty)$ and we complete the proof of locally Lipschitz continuity in $x$.

\vspace{0.5cm}
We turn to the proof of $t$-local Lipschitz continuity. For $0<t<\tau$ and each $i$, by \eqref{cali1},
\begin{equation}\label{leq}
u_{i}(x,\tau)-u_{i}(x,t)\leq-\int^{\tau}_{t}P_{i}(x,\mathbf{u}(x,s))ds\leq P(\tau)\cdot|\tau-t|,
\end{equation}
where we use the constant curve $\gamma(s)\equiv x,s\in[t,\tau]$ to connect $(x,t)$ with $(x,\tau)$.

\vspace{0.5cm}
For the other side of the inequality, let $\xi_{1}:[0,\tau]\rightarrow M$ be a minimizer of $u_{1}(x,\tau)$, i.e.,
\begin{equation*}
u_{1}(x,\tau)=\varphi_{1}(\xi_{1}(0))+\int^{\tau}_{0}L_{1}(\xi_{1}(s),\dot{\xi}_{1}(s),\mathbf{u}(\xi_{1}(s),s))ds.
\end{equation*}

By locally Lipschitz continuity in $x$, we have
\begin{align*}
&u_{1}(x,\tau)-u_{1}(\xi_{1}(t),t)\\
=&(u_{1}(x,\tau)-u_{1}(\xi_{1}(t),\tau))+(u_{1}(\xi_{1}(t),\tau)-u_{1}(\xi_{1}(t),t))\\
\leq&\kappa_{\varphi}(\tau)\cdot|x-\xi_{1}(t)|+P(\tau)\cdot|\tau-t|.
\end{align*}

By the above estimate and H\"{o}lder inequality, we have
\begin{align*}
|x-\xi_{1}(t)|^{r^{\ast}}&\leq\bigg[\int_{t}^{\tau}|\dot{\xi}_{1}(s)|ds\bigg]^{r^{\ast}}\leq\int_{t}^{\tau}|\dot{\xi}_{1}(s)|^{r^{\ast}}ds\cdot(\tau-t)^{\frac{r^{\ast}}{r}}\\
\leq&A\int_{t}^{\tau}l_{1}(\xi_{1}(s),\dot{\xi}_{1}(s))ds\cdot(\tau-t)^{\frac{r^{\ast}}{r}}\\
\leq&A\bigg[|u_{1}(x,\tau)-u_{1}(\xi_{1}(t),t)|+\int_{t}^{\tau}|P_{1}(\xi_{1}(s),\mathbf{u}(\xi_{1}(s),s))|ds\bigg]\cdot(\tau-t)^{\frac{r^{\ast}}{r}}\\
\leq&A[\kappa_{\varphi}(\tau)\cdot|x-\xi_{1}(t)|+2P(\tau)\cdot(\tau-t)]\cdot(\tau-t)^{\frac{r^{\ast}}{r}}.\\
\end{align*}
Thus we obtain
\begin{equation*}
|x-\xi_{1}(t)|^{\frac{r^{\ast}}{r}+1}\leq A\kappa_{\varphi}(\tau)\cdot|x-\xi_{1}(t)|\cdot|\tau-t|^{\frac{r^{\ast}}{r}}+2AP(\tau)\cdot|\tau-t|^{\frac{r^{\ast}}{r}},
\end{equation*}
which implies that there exists $c>0$ such that $|x-\xi_{1}(t)|\leq c|\tau-t|$.

\vspace{0.5cm}

By using the above estimate and Lemma \ref{x-Lip lem}, we obtain
\begin{equation}\label{geq}
\begin{split}
&u_{1}(x,\tau)-u_{1}(x,t)\\
=&(u_{1}(x,\tau)-u_{1}(\xi_{1}(t),t))+(u_{1}(\xi_{1}(t),t)-u_{1}(x,t))\\
\geq&\int^{\tau}_{t}L_{1}(\xi_{1}(s),\dot{\xi}_{1}(s),\mathbf{u}(\xi_{1}(s),s))ds-c\kappa_{\varphi}(\tau)\cdot|\tau-t|\\
\geq&-\int^{\tau}_{t}P_{1}(\xi_{1}(s),\mathbf{u}(\xi_{1}(s),s))ds-c\kappa_{\varphi}(\tau)\cdot|\tau-t|\\
\geq&-(P(\tau)+c\kappa_{\varphi}(\tau))\cdot|\tau-t|.
\end{split}
\end{equation}
Combining \eqref{leq} and \eqref{geq}, we complete the proof of $t$-locally Lipschitz continuity for $u_{1}$ and the same estimate also holds for $u_{2},...,u_{m}$.

\vspace{0.5cm}

Finally, we note that
$$
|u_{i}(x,t)-u_{i}(y,\tau)|\leq|u_{i}(x,\tau)-u_{i}(y,\tau)|+|u_{i}(x,t)-u_{i}(x,\tau)|,
$$
which reduces the proof of locally Lipschitz continuity to the proof of both $x$-locally Lipschitz continuous and $t$-locally Lipschitz continuous. \qed

\section{Further results}

Now we verify the results listed in Corollary \ref{cur lip11} one by one.

\vspace{1em}

\textit{Proof of (i): } We note that, for each $1\leq i\leq m$, $u_{i}:M\times[0,\infty)\rightarrow \mathbb{R}$ is a viscosity solution associated to the Hamilton-Jacobi equation
\begin{equation}
\left\{
  \begin{array}{ll}
    \partial_{t}u+H_{i}(x,\partial_{x}u,\mathbf{u}(x,t))=0, & \hbox{$(x,t)\in M\times[0,\infty)$;} \\
    u(0,x)=\varphi_{i}(x), & \hbox{$x\in M$.}
  \end{array}
\right.
\end{equation}

By locally Lipschitz continuity of $\mathbf{u}$ on $M\times(0,\infty)$, $H_{i}(x,p,\mathbf{u}(x,t))$ is locally Lipschitz continuous on $M\times(0,\infty)$; by condition (H1), $H_{i}(x,p,\mathbf{u}(x,t))$ is strictly convex with respect to $p$. Now we use Theorem \cite[Theorem 5.3.8]{CS} to complete the proof.\qed

\vspace{1em}

\textit{Proof of (ii): }
Fix any $i$ and $\tau\in(0,t]$. By locally Lipschitz continuity of $u_{i}$, there exists $\kappa>0$ and $h>0$ small enough such that for any $\tau-h\leq t_{1}<t_{2}\leq\tau+h$,
\begin{equation}\label{eq:7}
|u_{i}(\xi_{i}(t_{2}),t_{2})-u_{1}(\xi_{i}(t_{1}),t_{1})|\leq\kappa(|\xi_{i}(t_{2})-\xi_{i}(t_{1})|+|t_{2}-t_{1}|).
\end{equation}

On the other hand, $\mathbf{u}(x,s)$ belongs to the compact set $\mathcal{K}:=[-U(t),U(t)]^{m}$ for any $(x,s)\in M\times[0,t]$. So for $C>2\kappa$, by (L2), there exists $D>0$ such that for any $(x,s)\in M\times[0,t]$,
\begin{equation*}
L(x,\dot{x},\mathbf{u}(x,s))\geq C|\dot{x}|-D.
\end{equation*}

By Proposition \ref{ca} and the above inequality,
\begin{align*}
&u_{i}(\xi_{i}(t_{2}),t_{2})-u_{i}(\xi_{i}(t_{1}),t_{1})\\
=&\int^{t_{2}}_{t_{1}}L_{i}(\xi_{i}(s),\dot{\xi}_{i}(s),\mathbf{u}(\xi_{i}(s),s))ds\\
\geq&\int^{t_{2}}_{t_{1}}\bigg[C|\dot{\xi}_{i}(s)|-D\bigg]ds\geq C|\xi_{i}(t_{2})-\xi_{i}(t_{1})|-D|t_{2}-t_{1}|.
\end{align*}

Combining the above inequality and \eqref{eq:7}, we obtain that
\begin{equation}
\kappa|\xi_{i}(t_{2})-\xi_{i}(t_{1})|\leq(C-\kappa)|\xi_{i}(t_{2})-\xi_{i}(t_{1})|\leq(\kappa+D)|t_{2}-t_{1}|,
\end{equation}
this implies the locally Lipschitz continuity near $\tau$.\qed

\vspace{1em}

\textit{Proof of (iii): }
Again we fix $i$, we shall only show $\xi_{i}$ has left derivative satisfying \eqref{dual11}, since the proof is completely similar for other cases.

\vspace{1em}

By (ii), for any $0<h<t$, $\dot{\xi}_{i}|_{[t-h,t]}\in L^{\infty}([t-h,t],TM)$, i.e. there exists $R>0$ such that $|\dot{\xi}_{i}(s)|\leq R$ for any $s\in[t-h,t]$. Thus the difference $\frac{\xi_{i}(t)-\xi_{i}(t-h)}{h}$ is bounded and we denote the limit points of this difference by $\Omega$. It is clear that $\Omega$ is a bounded subset of $T_{x}M$.

\vspace{1em}

Since $u_{i}$ is differentiable at $(x,t)$ and $\xi_{i}$ is locally Lipschitz continuous at $t$, we have
\begin{equation}\label{diff}
\begin{split}
&\frac{u_{i}(x,t)-u_{i}(\xi_{i}(t-h),t-h)}{h}\\
=&\partial_{t}u_{i}(x,t)+P\cdot\frac{\xi_{i}(t)-\xi_{i}(t-h)}{h}+\frac{o(h+|\xi_{i}(t)-\xi_{i}(t-h)|)}{h}\\
=&\partial_{t}u_{i}(x,t)+P\cdot\frac{\xi_{i}(t)-\xi_{i}(t-h)}{h}+\frac{o(h)}{h}.
\end{split}
\end{equation}

On the other hand, by the boundedness of $\dot{\xi}_{i}$ near $t$ and the locally Lipschitz continuity of $\mathbf{u}$ at $(x,t)$, there exists $\kappa>0$ independent of $h$ such that
\begin{equation}\label{conv}
\begin{split}
&\frac{u_{i}(x,t)-u_{i}(\xi_{i}(t-h),t-h)}{h}\\
=&\frac{1}{h}\int^{t}_{t-h}L_{i}(\xi_{i}(s),\dot{\xi}_{i}(s),\mathbf{u}(\xi_{i}(s),s))ds\\
\geq&\frac{1}{h}\int^{t}_{t-h}\bigg[L_{i}(x,\dot{\xi}_{i}(s),\mathbf{u}(x,t))-\kappa(|\xi_{i}(t)-\xi_{i}(s)|+|t-s|)\bigg]ds\\
=&\frac{1}{h}\int^{t}_{t-h}L_{i}(x,\dot{\xi}_{i}(s),\mathbf{u}(x,t))ds-\frac{O(h^{2})}{h}\\
\geq&L_{i}(x,\frac{1}{h}\int^{t}_{t-h}\dot{\xi}_{i}(s)ds,\mathbf{u}(x,t))+O{(h)}\\
=&L_{i}(x,\frac{\xi_{i}(t)-\xi_{i}(t-h)}{h},\mathbf{u}(x,t))+O{(h)},
\end{split}
\end{equation}
where in the first inequality, we use that $\frac{\partial L_{i}}{\partial x}$ is bounded on compact sets of $TM\times\mathbb{R}^{m}$ and in the second inequality, we use the convexity of $L_{i}$ in $\dot{\xi}$.

For any $V\in\Omega$, we choose $h_{n}\rightarrow0$ such that $V=\frac{\xi_{i}(t)-\xi_{i}(t-h_{n})}{h_{n}}$. Now take $h=h_{n}$ in \eqref{diff} and \eqref{conv} and let $n$ goes to infinity, we obtain
\begin{equation*}
P\cdot V-L_{1}(x,V,\mathbf{u}(x,t))\geq H_{1}(x,P,\mathbf{u}(x,t)).
\end{equation*}
By (H1),(L1) and the definition of Legendre transformation, $\Omega$ is a singleton and the equation \eqref{dual11} holds.\qed

\vspace{0.5cm}

We end this section by the following remark:
\begin{remark}
From the proofs above, if we obtain the locally Lipschitz continuity of the viscosity solution, then (L1)-(L3) is sufficient for the validity of Corollary \ref{cur lip11}. Namely, the additional assumption (L*) is not necessary for this corollary.
\end{remark}

\textbf{Acknowledgements}
All authors would like thank Professor P. Cannarsa and Professor W. Cheng  for very helpful conversations on this topic. They also warmly thank Professor H. Ishii for his deep insights into partial results in this paper from PDE aspects, especially pointing out the relation between Theorem \ref{Lip1} and Lions' regularizing effect.

\end{document}